\documentclass[11pt]{article}

\usepackage{amssymb}
\usepackage{amsmath}
\usepackage{amsfonts}
\usepackage{graphics}
\usepackage{epsfig}
\usepackage{enumerate}

\usepackage{color}

\allowdisplaybreaks

\newtheorem{theorem}{Theorem}[section]
\newtheorem{lemma}[theorem]{Lemma}
\newtheorem{definition}[theorem]{Definition}

\newtheorem{example}[theorem]{Example}

\newenvironment{proof}{{\bf Proof:}}{~\hfill $\Box$}
\newenvironment{keywords}{{\bf Keywords: }}{}
\newenvironment{AMS}{{\bf AMS Subject Classification: }}{}

\numberwithin{equation}{section}

\begin{document}

\title{The Tychonoff uniqueness theorem for the $G$-heat equation
}
\author{LIN Qian  $^{1,2}$\footnote{ Corresponding address: Laboratoire de Math\'ematiques,
Universit\'{e} de Bretagne Occidentale, 6, avenue Victor Le Gorgeu,
CS 93837, 29238 Brest cedex 3, France.   {\it Email address}:
Qian.Lin@univ-brest.fr}
\\
{\small $^1$School of Mathematics, Shandong University, Jinan
250100,    China;}\\
{\small $^2$ Laboratoire de Math\'ematiques, CNRS UMR 6205,
Universit\'{e} de Bretagne Occidentale,}\\ {\small 6, avenue Victor
Le Gorgeu, CS 93837, 29238 Brest cedex 3, France.}  }

\date{}
\maketitle

\begin{abstract}
In this paper  we shall investigate a uniqueness result for
solutions of  the $G$-heat equation. We obtain the Tychonoff
uniqueness theorem for the $G$-heat equation.
\end{abstract}

\noindent

\vskip 3mm \noindent

\section{Introduction}

 Motivated by  the volatility uncertainty  problems, risk  measures  and
superhedging in finance,  Peng has introduced recently  a new notion
of a nonlinear expectation, the so-called $G$-expectation (see
\cite{Peng:2006},  \cite{Peng:2007}, \cite{Peng:2008},
\cite{Peng:2010}),   which is generated by the following  nonlinear
heat equation, called the $G$-heat equation:
\begin{eqnarray*}\label{}
\dfrac{\partial u}{\partial t}=G(D^{2}u),\quad (t, x)\in (0,
T)\times\mathbb{R}^{n},
\end{eqnarray*}
where $D^{2}u$ is the Hessian matrix of $u$, i.e.,
$D^{2}u=(\partial^{2}_{x_{i}x_{j}} u)_{i,j=1}^{n}$ and
\begin{eqnarray}\label{e4}
G(A)=\dfrac{1}{2}\sup\limits_{\alpha\in\Gamma}tr[\alpha\alpha^{T}A],
\quad  A=(A_{ij})_{i,j=1}^{n}\in \mathbb{S}^{n},
\end{eqnarray}
where $\mathbb{S}^{n}$ denotes the space of $n\times n $ symmetric
matrices and $\Gamma$ is a given non-empty, bounded and closed
subset of $\mathbb{R}^{n\times n}$, the space of $n\times n $
matrices.

 Together with the notion of the
$G$-expectation Peng (see \cite{Peng:2006}, \cite{Peng:2008},
\cite{Peng:2010}) introduced the related $G$-normal distribution and
the $G$-Brownian motion,  and established an It\^{o} calculus for
the $G$-Brownian motion. Peng  (see \cite{Peng2007},
\cite{Peng2008-1} and \cite{Peng:2010a}) also obtained the law of
large numbers and central limit theorem under nonlinear
expectations, which indicates that the notion of $G$-normal
distribution plays an important role in the theory of nonlinear
expectations as that of normal distribution in the classical
probability theory. The $G$-expectation can be regarded as a
coherent risk measure and the conditional $G$-expectation can be
regarded as a dynamic risk measure.

 Tychonoff (see  \cite{TY}, \cite{W} and Theorem 4.3.3 in \cite{KS}) obtained
  the following  uniqueness theorem for the heat equation.
\begin{theorem}\label{}
 Let $u_{1}, u_{2} \in C([0,T]\times \mathbb{R})$ be
solutions of the heat equation:
\begin{equation*}
\dfrac{\partial u}{\partial t}(t,x)=\Delta u(t,x),\quad (t, x)\in(0,
T)\times\mathbb{R}^{n},
\end{equation*}
with $u_{1}(0,x)=u_{2}(0,x)=\varphi(x).$ If there are two positive
constants $c_{1}, c_{2}$ such that
\begin{equation*}
|u_{1}(t,x)|\leq c_{1}e^{c_{2}|x|^{2}}, |u_{2}(t,x)|\leq
c_{1}e^{c_{2}|x|^{2}}, \ \text{uniformly for } t \in [0,T],
\end{equation*}
 then $u_{1}\equiv u_{2}$ in $[0,T]\times \mathbb{R}^{n}$.
\end{theorem}

The objective of this paper is to investigate the Tychonoff
uniqueness theorem for the following  generalized $G$-heat equation.
We also call it the $G$-heat equation.
\begin{eqnarray}\label{e1}
     u_{t}-G(t, x, D^{2}u)=0,\  (t,x) \in (0,T]\times
\mathbb{R}^{n},
\end{eqnarray}
where $G:[0,T]\times
\mathbb{R}^{n}\times\mathbb{S}^{n}\rightarrow\mathbb{R}$ and only
satisfies the following conditions:\\

(H) $G$ is continuous, and for all $(t,x)\in [0,T]\times
\mathbb{R}^{n}$, $G(t,x, \cdot)$ is  subadditive and uniformly elliptic, $G(t,x, 0)=0$.\\

Da Lio and Ley \cite{DL} obtained a uniqueness theorem for
second-order Bellman-Isaacs equations under quadratic growth
assumptions.  Peng \cite{Peng:2010a} obtained a uniqueness theorem
for a class of second order parabolic equations under the polynomial
growth condition. Str\"{o}mberg \cite{S} considered the Cauchy
problem for parabolic Isaacs's equations:
\begin{eqnarray}\label{e2}
\left\{\begin{array}{l l}
    u_{t}+F(t, x, u, Du, D^{2}u)=0, \  (t,x) \in (0,T)\times
\mathbb{R}^{n},\\
  u(0,x)=\varphi(x), \  x \in  \mathbb{R}^{n},
         \end{array}
  \right.
\end{eqnarray}
where
\begin{eqnarray*}
F(t, x, r, p,
X)=\sup\limits_{\gamma}\inf\limits_{\delta}[-tr(A^{\gamma,\delta}(t,x)X)+\langle
b^{\gamma,\delta}(t,x),p\rangle+c^{\gamma,\delta}(t,x)r-f^{\gamma,\delta}(t,x)],
\\ (t,x,r,p,X) \in
(0,T]\times\mathbb{R}^{n}\times\mathbb{R}\times\mathbb{R}^{n}\times\mathbb{S}^{n}.
\end{eqnarray*}
Str\"{o}mberg obtained the following  uniqueness of viscosity
solution of (\ref{e2}):

\begin{theorem}\label{}
Let some conditions be satisfied and let $u_{1}, u_{2} \in
C(\overline{Q})$ be solutions of (\ref{e2}) in the strip
$Q=(0,T)\times \mathbb{R}^{n}$ with
$u_{1}(0,x)=u_{2}(0,x)=\varphi(x).$ If there are two positive
constants $c_{1}, c_{2}$ such that
\begin{equation*}\label{}
|u_{1}(t,x)|\leq c_{1}e^{c_{2}|x|}, |u_{2}(t,x)|\leq
c_{1}e^{c_{2}|x|}, \ \text{uniformly for } t \in [0,T],
\end{equation*}
 then $u_{1}\equiv u_{2}$ in $\overline{Q}$.
\end{theorem}

Str\"{o}mberg raised a question if solutions of the Cauchy problem
for (\ref{e2}) under weaker conditions than $|u(t,x)|\leq Ke^{k|x|}$
are unique. We shall give a positive answer and  prove that
solutions of the Cauchy problem for the $G$-heat equation (\ref{e1})
satisfying $|u(t,x)|\leq Ke^{k|x|^{2}}$ are unique.

 The rest of the paper is organized as follows. In Section 2, we give
notations and preliminaries which will be needed in what follows. In
Section 3, we investigate the Tychonoff uniqueness theorem for the
$G$-heat equation.

\section{Notations and Preliminaries }
The objective of this section is to give some notations and
preliminaries, which we will need.  We first recall the definition
of the parabolic superjet and the parabolic subjet.

\begin{definition}   Let $u: (0,T)\times
\mathbb{R}^{n}\rightarrow\mathbb{R}$ and $(t,x)\in (0,T)\times
\mathbb{R}^{n}$. Then we define the parabolic superjet of $u$ at
$(t,x)$:
\begin{eqnarray*}
{\cal{P}}^{2,+}u(t,x)&=&\Big\{(p,q,X)\in \mathbb{R}\times
\mathbb{R}^{n}\times \mathbb{S}^{n}\
 | \ u(s,y)\leq u(t,x)+p(s-t)+\langle q,y-x\rangle\\
 &&\quad+\displaystyle\frac{1}{2}\langle X(y-x),y-x \rangle
 +o(|s-t|+\|y-x\|^{2}),\ (s,y) \rightarrow(t,x)\Big\},
\end{eqnarray*} and its closure:
\begin{eqnarray*}
\bar{{\cal{P}}}^{2,+}u(t,x)&=&\Big\{(p,q,X)\in \mathbb{R}\times
\mathbb{R}^{n}\times \mathbb{S}^{n}\
 | \ \exists \ (t_{n},x_{n}, p_{n},q_{n},X_{n}) \ \text{such that}\\
&&\quad (p_{n},q_{n},X_{n})\in {\cal{P}}^{2,+}u(t_{n},x_{n})\
\text{and}\ (p,q,X)=\lim\limits_{n\rightarrow
+\infty}(p_{n},q_{n},X_{n}),\\ &&\quad \mbox{and}\
\lim\limits_{n\rightarrow
                  +\infty}(t_{n},x_{n},u(t_{n},x_{n}))=(t,x,u(t,x))\Big\}.
\end{eqnarray*}
\end{definition}
Similarly, we consider the parabolic subjet and its closure:
\begin{eqnarray*}
 {\cal{P}}^{2,-}u(t,x)=-{\cal{P}}^{2,+}(-u)(t,x),
\bar{{\cal{P}}}^{2,-}u(t,x)=-\bar{{\cal{P}}}^{2,+}(-u)(t,x).
\end{eqnarray*}
According to \cite{CIL1992}, we have
$$
\begin{array}{r}
{\cal{P}}^{2,+(-)}u(t,x)=\Big\{(\displaystyle\frac{\partial
\varphi}{\partial t}(t,x),D\varphi(t,x),D^{2}\varphi(t,x)),
\varphi\in C^{1,2}([0,T]\times \mathbb{R}^{n}),\\
u-\varphi \mbox{\ has a global maximum(minimum) 0 at }(t,x)\Big\}.
\end{array}
$$

We now recall the definition of  viscosity solution of (\ref{e1})
from Crandall, Ishii and Lions \cite{CIL1992}.
\begin{definition}$\ $
\begin{enumerate}[(i)]
\item A viscosity subsolution of  (\ref{e1}) on $(0,T)\times \mathbb{R}^n$
 is a function $u\in USC((0,T)\times \mathbb{R}^n)$
such that, for all $(t,x)\in (0,T)\times \mathbb{R}^n$,
\begin{eqnarray*}
p-G(t,x, X)\leq 0, \ for \ (p,q,X)\in\mathcal {P}^{2,+}u(t,x).
\end{eqnarray*}

\item  A viscosity supsolution of  (\ref{e1}) on $(0,T)\times \mathbb{R}^n$
 is a function $u\in LSC((0,T)\times \mathbb{R}^n)$
such that, for all $(t,x)\in (0,T)\times \mathbb{R}^n$,
\begin{eqnarray*}
p-G(t,x, X)\geq 0, \ for \ (p,q,X)\in\mathcal {P}^{2,-}u(t,x).
\end{eqnarray*}

\item $u\in C((0,T)\times \mathbb{R}^n)$ is said to be a
viscosity solution of (\ref{e1}) on $(0,T)\times \mathbb{R}^n$ if it
is both a viscosity subsolution and supersolution  of (\ref{e1}) on
$(0,T)\times \mathbb{R}^n$.
\end{enumerate}
\end{definition}

Let $M>0, x\in \mathbb{R}^{n}.$ We say that $P(x)$ is a paraboloid
of opening $M$ if $P(x)=\pm\frac{M}{2}|x|^{2}+l(x)+l_{0},$ where $l$
is linear and $l_{0}$ is a constant. $P(x)$ is convex if $+$ appears
and concave if $-$ appears. So for $t_{0}, \rho>0,$ the equation
$t=t_{0}-\frac{|x|^{2}}{\rho^{2}}$ denotes the graph of a concave
paraboloid of opening $\frac{2}{\rho^{2}}$ with vertex at $(t_{0},
0)\in \mathbb{R}^{n+1},$ which we will henceforth write as
$\rho=\frac{|x|}{\sqrt{t_{0}-t}}$. By concentric concave paraboloids
of opening $2\rho_{1}^{-2}$ and $2\rho_{2}^{-2}$, we mean these
paraboloids have common vertex $(t_{0}, 0)$.

Let $Q\subset\mathbb{R}^{n+1}$. $Q$ is bounded below by the line
$t=0$ and above by the line $t=t', \text{where} \ t'<t_{0}$. $Q$ is
bounded laterally by the arcs of the paraboloids
$\rho_{1}=\frac{|x|}{\sqrt{t_{0}-t}}$ and
$\rho_{2}=\frac{|x|}{\sqrt{t_{0}-t}}$ of opening $2\rho_{1}^{-2}$
and $2\rho_{2}^{-2}$ respectively, with $\rho_{1}\leq\rho_{2}.$
Geometrically, $Q$ is a concave paraboloid shell, truncated just
below the vertex $(t_{0}, 0).$  For $\rho_{1}\leq\rho\leq\rho_{2}$,
we define the functions as follows:

\begin{eqnarray*}
M_{1}(\rho)=\max\limits_{|x|=\rho\sqrt{t_{0}-t},\ 0 \leq t \leq
t'}u(t,x),\\
 M_{2}=\max\limits_{\rho_{1}\sqrt{t_{0}}\leq|x|\leq\rho_{2}\sqrt t_{0}}u(0,x),\\
 M(\rho)=\max\{ M_{1}(\rho), M_{2}\}.
\end{eqnarray*}

For $f\in C(Q)$ and positive constants $\lambda\leq\Lambda$, we
denote by $\underline{S}(\lambda,\Lambda,f)$ the class of viscosity
subsolutions of the equation $\mathcal
{M}(D^{2}u,\lambda,\Lambda)-u_{t}=f(t,x)$, where for any real
$n\times n$ symmetric matrix $M$
\begin{eqnarray*}
\mathcal
{M}(D^{2}u,\lambda,\Lambda)=\Lambda\sum\limits_{e_{i}>0}e_{i}(M)+\lambda\sum\limits_{e_{i}<0}e_{i}(M),
\end{eqnarray*}
where $e_{i}(M)$ are the eigenvalue of $M$.

Finally, the following three curves Lemma was obtained in Kovats
\cite{JK}.
\begin{lemma}\label{le1}
Let $u\in \underline{S}(\lambda,\Lambda,0)$ in a domain
$Q\subset\mathbb{R}^{n+1}$ containing two concave concentric
paraboloids of opening $2\rho_{1}^{-2}$ and $2\rho_{2}^{-2}$ and the
region between them. If $M(\rho)$ denotes the maximum of $u$ on any
concentric concave paraboloid of opening $2\rho^{-2}$, with
$\rho_{1}\leq\rho\leq\rho_{2}$, then there exists a differential
function $\psi(\rho)$ depending on $n, \lambda, \Lambda$ and $\rho$,
such that
 \begin{eqnarray*}\label{}
M(\rho)\leq\dfrac{M(\rho_{1})(\psi(\rho_{2})-\psi(\rho))+M(\rho_{2})
(\psi(\rho)-\psi(\rho_{1}))}{\psi(\rho_{2})-\psi(\rho_{1})}.
\end{eqnarray*}
Moreover, if $\psi'(\rho)\geq0$,  then
$\psi'(\rho)=e^{\frac{\rho^{2}}{4\Lambda}}\rho^{-\frac{\Lambda(n-1)}{\lambda}}$;
and if $\psi'(\rho)\leq0$,  then
$\psi'(\rho)=-e^{\frac{\rho^{2}}{4\lambda}}\rho^{-\frac{\lambda(n-1)}{\Lambda}}$.
\end{lemma}

\section{Main result}

 The objective of this section is to investigate  the Tychonoff uniqueness theorem for the $G$-heat
equation. In order to get this theorem, we first give the following
lemma.

\begin{lemma}\label{2}
Let $u$ be a viscosity subsolution and $v$ be a viscosity
supersolution of (\ref{e1}).  Then   $u-v$ is a viscosity
subsolution of (\ref{e1}).
\end{lemma}

\begin{proof}
   Let $\varphi \in C^{2}((0,T]  \times \mathbb{R}^{n})$ and $(t_{0},
x_{0}) \in [0,T]\times \mathbb{R}^{n}$ be a strict local maximum
point of $u-v-\varphi$, and more precisely a strict maximum point in
$[t_{0}-r, t_{0}+r]\times \overline{B}(x_{0},r)$, where
$\overline{B}(x_{0},r)=\{x\in\mathbb{R}^{n}: \ |x-x_{0}|\leq r\}$ is
a ball with a radius $r>0$. Then we consider the function:
   $$\Phi_{\varepsilon}(t, x, y)=u(t, x)- v(t, y)- \dfrac{|x-y|^{2}}{\varepsilon^{2}}-\varphi(t, x),$$
 where $\varepsilon$ is a positive parameter.

 Since $(t_{0}, x_{0}) $ is a strict  maximum point of
 $u-v-\varphi$ in $[t_{0}-r, t_{0}+r]\times \overline{B}(x_{0},r)$,
  then by virtue of the classical argument  of viscosity
 solutions, there exists   $(t_{\varepsilon}, x_{\varepsilon},y_{\varepsilon})$ such that

\begin{enumerate}[(i)]
\item $(t_{\varepsilon}, x_{\varepsilon},y_{\varepsilon})$ is a strict maximum
point of $\Phi_{\varepsilon}(t, x, y)$ in $[t_{0}-r, t_{0}+r]\times
\overline{B}(x_{0},r)\times \overline{B}(x_{0},r)$;

\item $(t_{\varepsilon}, x_{\varepsilon},y_{\varepsilon}) \rightarrow (t_{0}, x_{0},  x_{0})
 $, as $\varepsilon\rightarrow 0$;

\item $\dfrac{| x_{\varepsilon}-y_{\varepsilon}|^{2}}{\varepsilon^{2}}$ is bounded
  and $\dfrac{|x_{\varepsilon}-y_{\varepsilon}|^{2}}{\varepsilon^{2}}\rightarrow
  0$, as $\varepsilon\rightarrow 0$.
\end{enumerate}
 Thanks to Theorem 8.3 in \cite{CIL1992}, for every $\alpha>0$, there exist $p,q\in\mathbb{R}$ and $X, Y \in
  \mathbb{S}^{n}$ such that
\begin{eqnarray*}
&&\Big(p,\dfrac{2(x_{\varepsilon}-y_{\varepsilon})}{\varepsilon^{2}}+D\varphi(t_{\varepsilon},
 x_{\varepsilon}), X\Big)\in \overline{\mathcal {P}}^{2, +}u(t_{\varepsilon},
 x_{\varepsilon}),\\
&&\Big(q,\dfrac{2(x_{\varepsilon}-y_{\varepsilon})}{\varepsilon^{2}},
Y\Big)
 \in \overline{\mathcal {P}}^{2, -}v(t_{\varepsilon},
 y_{\varepsilon}),\\
&& p-q=\dfrac{\partial\varphi(t_{\varepsilon},
 x_{\varepsilon})}{\partial t},
\end{eqnarray*}
and
\begin{eqnarray*}\label{}
 \left(
\begin{array}{cc}
X & 0 \\
0 &-Y
\end{array}
\right)\leq A+\alpha A^{2},
\end{eqnarray*}
where
\begin{eqnarray*}
 A=
\left(
\begin{array}{cc}
D^{2}\varphi(t_{\varepsilon},
 x_{\varepsilon})+ \dfrac{2I}{\varepsilon^{2}}& -\dfrac{2I}{\varepsilon^{2}},
 \\
-\dfrac{2I}{\varepsilon^{2}} & \dfrac{2I}{\varepsilon^{2}}
\end{array}
\right).
\end{eqnarray*}
Taking $\alpha=\dfrac{\varepsilon^{2}}{2}$, we get
\begin{eqnarray*}
 \left(
\begin{array}{cc}
X & 0 \\
0 &-Y
\end{array}
\right)\leq \dfrac{6}{\varepsilon^{2}} \left(
\begin{array}{cc}
I& -I \\
-I &I
\end{array}
\right) + \dfrac{\varepsilon^{2}}{2} \left(
\begin{array}{cc}
(D^{2}\varphi(t_{\varepsilon},
 x_{\varepsilon}))^{2}& 0 \\
0 & 0
\end{array}
\right) \nonumber \\  +\left(
\begin{array}{cc}
3D^{2}\varphi(t_{\varepsilon},
 x_{\varepsilon})& -D^{2}\varphi(t_{\varepsilon},
 x_{\varepsilon}) \\
-D^{2}\varphi(t_{\varepsilon},
 x_{\varepsilon}) & 0
\end{array}
\right).
\end{eqnarray*}
Therefore, we have
\begin{eqnarray}\label{e3}
X-Y\leq \dfrac{\varepsilon^{2}}{2}(D^{2}\varphi(t_{\varepsilon},
 x_{\varepsilon}))^{2}+D^{2}\varphi(t_{\varepsilon},
 x_{\varepsilon}).
\end{eqnarray}
 Since $u$ is a viscosity subsolution and $v$ is
a viscosity supersolution of (\ref{e1}), then we have
\begin{eqnarray*}\label{}
 p-G(t_{\varepsilon},
 x_{\varepsilon},X)\leq 0,\quad q-G(t_{\varepsilon},
 x_{\varepsilon},Y)\geq 0,
\end{eqnarray*}
and the above inequality and the subadditivity of
$G(t_{\varepsilon},
 x_{\varepsilon}, \cdot)$
yield
\begin{eqnarray*}\label{}
\dfrac{\partial\varphi(t_{\varepsilon},
 x_{\varepsilon})}{\partial t}= p-q\leq G(t_{\varepsilon},
 x_{\varepsilon},X)-G(t_{\varepsilon},
 x_{\varepsilon},Y)\leq G(t_{\varepsilon},
 x_{\varepsilon},X-Y).
\end{eqnarray*}
By the above inequality and (\ref{e3}) we have
 $$\dfrac{\partial\varphi(t_{\varepsilon},
 x_{\varepsilon})}{\partial t}\leq G(t_{\varepsilon},
 x_{\varepsilon}, X-Y)\leq G\Big(t_{\varepsilon},
 x_{\varepsilon}, \dfrac{\varepsilon^{2}}{2}(D^{2}\varphi(t_{\varepsilon},
 x_{\varepsilon}))^{2}+D^{2}\varphi(t_{\varepsilon},
 x_{\varepsilon})\Big).$$
 Letting $\varepsilon\rightarrow 0$, since
  $(t_{\varepsilon}, x_{\varepsilon},y_{\varepsilon}) \rightarrow (t_{0}, x_{0},  x_{0})
 $, as $\varepsilon\rightarrow 0$, and $G$  is continuous, we get
  $$\dfrac{\partial\varphi(t_{0},
 x_{0})}{\partial t}- G\big(t_{0},
 x_{0}, D^{2}\varphi(t_{0},
 x_{0})\big)\leq 0.$$
Therefore,  $u-v$ is a viscosity subsolution of (\ref{e1}). The
proof is complete.
\end{proof}\\

We now state and prove the main result in this paper.
\begin{theorem}\label{}
 Let (H) be satisfied and let $u_{1}, u_{2} \in C(\overline{Q})$ be
solutions of (\ref{e1}) in the strip $Q=(0,T)\times \mathbb{R}^{n}$
with $u_{1}(0,x)=u_{2}(0,x)=\varphi(x).$ If there are two positive
constants $c_{1}, c_{2}$ such that
\begin{equation}\label{8}
|u_{1}(t,x)|\leq c_{1}e^{c_{2}|x|^{2}}, |u_{2}(t,x)|\leq
c_{1}e^{c_{2}|x|^{2}}, \ \text{uniformly for } t \in [0,T],
\end{equation}
 then $u_{1}\equiv u_{2}$ in $\overline{Q}$.
\end{theorem}

\begin{proof}
Since $u_{1}$ is a viscosity subsolution and $u_{2}$ is a viscosity
supersolution of (\ref{e1}), then by Lemma \ref{2} we have
$v=u_{1}-u_{2}$ is a viscosity subsolution of (\ref{e1}) with
$v(0,x)=0$. Due to \cite{CC}, we have $v\in
\underline{S}(\frac{\lambda}{n},\Lambda,0)$, where
$\lambda\leq\Lambda$ are two positive constants.

Putting $t_{0}\leq \dfrac{1}{4\Lambda c_{2}}$, we first consider $v$
in a domain $Q_{1}=[0,\frac{t_{0}}{2}]\times \mathbb{R}^{n}$. For
$\rho_{1}\leq\rho\leq\rho_{2}$, from Lemma \ref{le1} there exists a
differential function $\psi(\rho)$ depending on $n, \lambda,
\Lambda$ and $\rho$, such that
\begin{eqnarray*}\label{}
M(\rho)\leq\dfrac{M(\rho_{1})(\psi(\rho_{2})-\psi(\rho))+M(\rho_{2})
(\psi(\rho)-\psi(\rho_{1}))}{\psi(\rho_{2})-\psi(\rho_{1})}.
\end{eqnarray*}
Moreover, if $\psi'(\rho)\geq0$,  then
$\psi'(\rho)=e^{\frac{\rho^{2}}{4\Lambda}}\rho^{-\frac{n\Lambda(n-1)}{\lambda}}$;
and if $\psi'(\rho)\leq0$,  then
$\psi'(\rho)=-e^{\frac{\rho^{2}}{4\lambda}}\rho^{-\frac{\lambda(n-1)}{n\Lambda}}$.

By (\ref{8}), we know that $ |v(t, x)|\leq 2c_{1}e^{c_{2}|x|^{2}}$.
Then $M(\rho_{2})\leq2c_{1}e^{c_{2}\rho_{2}^{2}t_{0}}$. If
$\psi'\geq 0$, then we have
\begin{eqnarray*}
 \lim\limits_{\rho_{2}\rightarrow
 \infty}\frac{M(\rho_{2})}{\psi(\rho_{2})}\leq\lim\limits_{\rho_{2}\rightarrow
 \infty}\frac{2c_{1}e^{c_{2}\rho_{2}^{2}t_{0}}}{\psi(\rho_{2})}=\lim\limits_{\rho_{2}\rightarrow
 \infty}\frac{4c_{1}c_{2}\rho_{2}t_{0}e^{c_{2}\rho_{2}^{2}t_{0}}}
 {e^{\frac{\rho_{2}^{2}}{4\Lambda}}\rho^{-\frac{n\Lambda(n-1)}{\lambda}}}
 =\lim\limits_{\rho_{2}\rightarrow
 \infty}\frac{4c_{1}c_{2}t_{0}\rho_{2}^{1+\frac{n\Lambda(n-1)}{\lambda}}}
 {e^{\frac{\rho_{2}^{2}}{4\Lambda}-c_{2}\rho_{2}^{2}t_{0}}}.
\end{eqnarray*}
Since $t_{0}\leq \dfrac{1}{4\Lambda c_{2}}$, we have
\begin{eqnarray*}
 \lim\limits_{\rho_{2}\rightarrow
 \infty}\frac{M(\rho_{2})}{\psi(\rho_{2})}\leq0.
\end{eqnarray*}
 Therefore,
\begin{eqnarray*}
M(\rho)&\leq&\lim\limits_{\rho_{2}\rightarrow \infty}
\dfrac{M(\rho_{1})(\psi(\rho_{2})-\psi(\rho))+M(\rho_{2})
(\psi(\rho)-\psi(\rho_{1}))}{\psi(\rho_{2})-\psi(\rho_{1})}\\
&\leq&M(\rho_{1})+\lim\limits_{\rho_{2}\rightarrow \infty}
\dfrac{M(\rho_{1})(\psi(\rho_{1})-\psi(\rho))+M(\rho_{2})
(\psi(\rho)-\psi(\rho_{1}))}{\psi(\rho_{2})-\psi(\rho_{1})}\\
&\leq & M(\rho_{1}).
\end{eqnarray*}
Letting $\rho_{1}\rightarrow 0$, we know that the maximum value of
$v$ in $Q_{1}$ occurs on the hyperplane $x=0$. If the maximum value
of $v$ in $Q_{1}$ occurs at $(x=0, t=0)$, then $v\leq v(0,0)=0$ in
$Q_{1}$. We consider $-v$ in $Q_{1}$. By the similar argument, we
have $-v\leq -v(0,0)=0$ in $Q_{1}$. Thus,  $v=0$ in $Q_{1}$.
 If the maximum value
of $v$ in $Q_{1}$ occurs at $(x=0, t=t_{1})$, where $t_{1}\in (0,
\frac{t_{0}}{2}]$, then by the strong maximum principle in  \cite{G}
we have $v= v(0,0)=0$ in $[0,t_{1}]\times \mathbb{R}^{n}$.

Repeating the above process, using $t=t_{1}$ as the initial line and
there exists $t_{2}$ such that $t_{1}< t_{2}\leq t_{0}$,  we obtain
that $v\equiv 0$ in $Q_{2}=[t_{1}, t_{2}]\times \mathbb{R}^{n}$.
After a finite number of steps, we get $v\equiv 0$ in $Q$.
Therefore, $u_{1}\equiv u_{2}$ in $Q$.

 If $\psi'\leq 0$, by the similar argument we can get the desired result. The proof is complete.
\end{proof}\\

  We give a counterexample (see \cite{J}) to show that if (\ref{8}) is not satisfied,
  then solution of the heat equation are not unique.
\begin{example}
We consider the following heat equation:
\begin{eqnarray*}\label{}
\left\{\begin{array}{l l}
     u_{t}-u_{xx}=0,\  (t,x) \in (0,T]\times
\mathbb{R}^{n},\\
      u(0,x)=0,  \  x \in  \mathbb{R}^{n}.
         \end{array}
  \right.
\end{eqnarray*}
The above equation has many solutions. In fact, for any $\alpha>1$,
putting
\begin{equation*}\label{}
g(t)=
\begin{cases}
e^{-t^{-\alpha}},& t>0;\\
0, &otherwise,
\end{cases}
\end{equation*}
we can check that
$$u(t,x)=\sum_{k=0}^\infty\frac{g^{(k)}(t)x^{2k}}{2k!}.$$
are solutions of the above heat equation.
\end{example}

\section*{Acknowledgements}
The  author  thanks  Prof. Peng and Dr. Jia  for their helpful
discussions.  The author also thanks  the editor and anonymous
referees for their  helpful suggestions.

This work is supported by the Young Scholar Award for Doctoral
Students of the Ministry of Education of China, the Marie Curie
Initial Training Network (PITN-GA-2008-213841) and the National
Basic Research Program of China (973 Program, No. 2007CB814900).

\end{document}